\documentclass[10pt,a4paper,twoside]{article}
\usepackage{amssymb,amsthm}
\usepackage{amsmath}
\usepackage{setspace}
\usepackage[a4paper,top=33mm, bottom=27mm, left=20mm, right=20mm]{geometry}
\usepackage[small, margin=20pt]{caption}
\usepackage{url}
\usepackage{color}
\usepackage{graphicx}
\usepackage{hyperref}

\usepackage{fancyhdr}
\allowdisplaybreaks[4]

\title{Correlation among runners and some results on the Lonely Runner Conjecture}
\author{Guillem Perarnau\thanks{Corresponding author}\phantom{.} and Oriol Serra\\
\phantom{.}\\
\emph{Departament de Matem\`atica Aplicada 4}\\
\emph{Universitat Polit\`ecnica de Catalunya, BarcelonaTech}\\
\emph{C/ Jordi Girona, 1-3. Edifici C3. 08034 Barcelona}}
\date{\today}

\theoremstyle{plain}
\newtheorem{theorem}{Theorem}
\newtheorem{lemma}[theorem]{Lemma}
\newtheorem{proposition}[theorem]{Proposition}
\newtheorem{corollary}[theorem]{Corollary}
\newtheorem{observation}[theorem]{Observation}

\theoremstyle{definition}
\newtheorem{conjecture}[theorem]{Conjecture}

\newtheorem*{acknowledgement}{Acknowledgement}

\theoremstyle{definition}

\renewenvironment{proof}[1][Proof]{\begin{trivlist}
\item[\hskip \labelsep {\textit{#1}.}]}{\qed\end{trivlist}}

\newenvironment{proofof}[1][Proofof]{\begin{trivlist}
\item[\hskip \labelsep {\textit{#1}}.]}{\qed\end{trivlist}}

\newcommand{\rst}[1]{\ensuremath{{\mathbin\upharpoonright}%
\raise-.5ex\hbox{$#1$}}} 

\newcommand{\E}{\mathbb{E}}
\newcommand{\eps}{\varepsilon}
\newcommand{\TT}{(0,1)}
\newcommand{\ZZ}{\mathbb{Z}}
\newcommand{\RR}{\mathbb{R}}

\begin{document}

\pagenumbering{arabic}

\setcounter{section}{0}

\maketitle

\onehalfspace

\thispagestyle{empty}

\begin{abstract}
The Lonely Runner Conjecture, posed independently by
Wills and by Cusick, states that for
any set of runners running along the unit circle with constant different speeds and starting at the
same point, there is a time where all of them are far enough from the origin. We study the
correlation among the time that runners spend close to the origin. By means of these correlations,
we improve a result of Chen on the gap of loneliness. In the last part, we introduce dynamic interval graphs to deal with
a weak version of the conjecture thus providing a new result related to the invisible runner theorem of
Czerwi{\'n}ski and Grytczuk.
\end{abstract}


\section{Introduction}

The Lonely Runner Conjecture was posed independently by Wills~\cite{w1967} in 1967    and
Cusick~\cite{c1982} in 1982. Its picturesque name comes from the following interpretation due to Goddyn \cite{bggst1998}.
Consider a set of $k$ runners  on the unit circle running with different constant speeds and
starting at the origin. The conjecture states that, for each runner, there is a time where she is at distance at least $1/k$ on the circle from all the other
runners.

For any real number $x$, denote by $\|x\|$, the distance from $x$ to the closest
integer
$$
\|x\| = \min \{x-\lfloor x\rfloor,  \lceil x\rceil -x\} \;,
$$
and by
$$
\{x\}=x-\lfloor x\rfloor,
$$
its fractional part.

By assuming that one of the runners has zero speed, the conjecture can be easily seen to be equivalent to the following one.

\begin{conjecture}[Lonely Runner Conjecture]\label{conj:LRC}
For every $n\geq 1$ and every set of nonzero speeds $v_1,\dots ,v_{n}$, there exists a time $t\in
\RR$ such that
$$
 \| tv_i\|\geq \frac{1}{n+1}\;,
$$
for every $i\in [n]$.
\end{conjecture}

If true, the Lonely Runner Conjecture is best possible: for the set of speeds,
\begin{align}\label{eq:speeds}
 v_i & = i \quad \text{for every $i\in [n]$}\;,
 \end{align}
there is no time for which all the runners are further away  from
the origin than $\tfrac{1}{n+1}$. An infinite family of additional extremal sets for the conjecture
can be found in~\cite{gw2006}.

The conjecture is obviously true for $n=1$, since at some point $\|
tv_1\|=1/2$, and it is also easy to show that it holds for $n=2$.
Many proofs for $n=3$ are given
in the context of diophantine approximation (see~\cite{bw1972,c1982}). A computer--assisted proof
for $n=4$ was given by Cusick and Pomerance motivated by a view-obstruction problem in
geometry~\cite{cp1984}, and later Biena et al.~\cite{bggst1998} provided a simpler proof by
connecting it to
nowhere zero flows in regular matroids.  The conjecture was proved for $n=5$  by Bohmann, Holzmann
and Kleitman~\cite{bhk2001}. Barajas and Serra~\cite{bs2008} showed that the conjecture holds
for $n=6$ by studying the regular chromatic number of distance graphs.

In~\cite{bhk2001}, the authors also showed that the conjecture can be reduced to the case where all
speeds are positive integers and in the sequel we will assume this to be the case. In particular,
we also may assume that $t$ takes values on the $(0,1)$ unit interval, since if
$t\in\ZZ$, then $\| tv_i\|=0$ for all $i\in [n]$.

Czerwi{\'n}ski~\cite{c2011}  proved a strengthening of the
conjecture if all the speeds are chosen uniformly at random among all the
$n$-subsets of $[N]$. In particular, Czerwi{\'n}ski's result implies that, for almost all sets of
runners, as $N\to \infty$ there is a time where
all the runners are arbitrarily close to $1/2$. The dependence of $N$ with respect
to $n$ for which this result is valid was improved  by Alon~\cite{a2013} in the context
of colorings of Cayley graphs.

Dubickas~\cite{d2011} used a result of Peres and Schlag~\cite{ps2010} in lacunary integer
sequences to prove that the conjecture holds if the sequence of increasing speeds grows fast enough;
in particular, for  $n$ sufficiently  large, if
\begin{align}\label{eq:dubi}
\frac{v_{i+1}}{v_i} \geq 1+\frac{22\log{n}}{n}\;,
\end{align}
for every $1\leq i<n$.
These results introduce the use of the Lov\'asz Local Lemma to deal with the dependencies
among the runners.

Another approach to the conjecture is to reduce the \emph{gap of loneliness}. That is, to show that, for some fixed $\delta\leq \frac{1}{n+1}$ and  every set of nonzero speeds, there exists a time $t\in
\TT$ such that
\begin{align}\label{eq:delta}
 \| tv_i\| &\geq \delta \quad\text{for every $i\in [n]$}\;.
\end{align}

For this approach it is particularly useful to define the following sets,
$$
A_i= \{t\in \TT :\; \|tv_i\|<\delta \}\;.
$$
For every $t\in A_i$, we will say that the $i$--th runner is $\delta$--close to the origin at time
$t$. Otherwise, we will say that the runner is $\delta$--far from the origin at time $t$.

The set $A_i$ can be thought of as an event in the probability space $\TT$
with the uniform distribution. Notice that we have $\Pr(A_i)=2\delta$  independently
from the value of $v_i$. In this setting, if
\begin{align}\label{eq:bad_events}
\Pr\left(\bigcap_{i=1}^n \overline{A_i}\right)>0\;,
\end{align}
then, there exists a time $t$ for which~\eqref{eq:delta} holds.

Here it is also convenient to consider the indicator random variables $X_i$ for the events $A_i$.
Let $X=\sum_{i=1}^n X_i$ count the number of runners which are $\delta$--close from the origin at a
time $t\in (0,1)$ chosen uniformly at random. Then, condition~\eqref{eq:bad_events} is equivalent to
$\Pr(X=0)>0$.

A first straightforward result in this direction is obtained by using the union bound
in~\eqref{eq:bad_events}. For any $\delta<\frac{1}{2n}$, we have
\begin{align*}
\Pr\left(\bigcap_{i=1}^n \overline{A_i}\right)&\geq 1- \sum_{i=1}^n \Pr(A_i) = 1-2\delta n>0\;.
\end{align*}

This result was improved by Chen~\cite{c1994} who showed that, for every set of $n$ nonzero speeds,
there
exists a time $t\in \RR$ such that
\begin{align}\label{eq:chen}
\| tv_i\|\geq\frac{1}{2n-1+\tfrac{1}{2n-3}}\;,
\end{align}
for every $i\in [n]$.

If $2n-3$ is a prime number, then the previous result was extended by Chen and Cusick~\cite{cc1999}.
In this case, these authors  proved that, for every set of $n$ speeds, there exists a time $t\in \RR$ such
that
  $$
\| tv_i\|\geq\frac{1}{2n-3}\;,
  $$
for every $i\in [n]$. We give a seemingly simpler proof of this result in Section~\ref{sec:consec_lonely}.  Unfortunately, both proofs strongly use the fact that $2n-3$ is a prime. 

In order to improve~\eqref{eq:chen}, we exactly compute the pairwise join probabilities
$\Pr(A_i\cap A_j)$, the amount of time that two runners spend close to the origin at the same time.
As a corollary, we give the following lower bound on $\E(X^2)$.
\begin{proposition}\label{prop:variance}
	For every $\delta$ such that $\delta\to 0$ when $n\to +\infty$, we have
	$$
	\E(X^2) \geq 2\delta n\left(\delta\left(1+\Omega\left(\frac{1}{\log{\delta^{-1}}}\right)\right)n+ 1\right)\;.
	$$
\end{proposition}
Then, we are able to improve Chen's result on the gap of loneliness around the origin.
\begin{theorem}\label{thm:2}
For every sufficiently large $n$ and every set $v_1,\dots
,v_{n}$ of nonzero speeds there exists a time  $t\in
\TT$ such that
$$
 \| tv_i\|\geq\frac{1}{2n-2+o(1)}\;,
$$
for each $i\in [n]$.
\end{theorem}
The proof of Theorem \ref{thm:2} uses a Bonferroni--type inequality due to Hunter \cite{hunter1976} (see Lemma~\ref{lem:bonf_local})
that improves the union bound with the knowledge of pairwise intersections. While the improvement of Theorem~\ref{thm:2} is modest, we point out that this is the best result up to date on the Lonely Runner Conjecture for a general sequence of speeds (provided that $n$ is large).

The bound on $\delta$ in Theorem \ref{thm:2} can be substantially improved in the case of sets of speeds taken from a sequence with divergent sum of
inverses. More precisely the following result is proven.

\begin{theorem}\label{thm:div} For every set  $v_1,\dots
,v_{n}$ of nonzero speeds there exists a time  $t\in
\TT$ such that
$$
 \| tv_i\|\geq\frac{1}{2\left(n-\sum_{i=2}^n \frac{1}{v_i}\right)}\;,
$$
for each $i\in [n]$.
\end{theorem}
The condition~\eqref{eq:dubi} of Dubickas~\cite{d2011} implies that the conjecture is true if the speeds grow sufficiently fast. Theorem~\ref{thm:div} is interesting in the sense that it provides meaningful bounds for the opposite case, that is when the speeds grow slowly. In particular, if $v_i, 1\leq i\leq n$ is a sequence of speeds satisfying $\sum_{i=1}^n
\frac{1}{v_i}=\omega(1)$, then there exists a time $t\in \TT$ such that
$$
 \| tv_i\|\geq\frac{1}{2n-\omega(1)}\;,
$$
for every $i\in [n]$. The last inequality holds under a natural density condition on the set of speeds which covers the more difficult cases where the speeds grow slowly.

Another interesting result on the Lonely Runner Conjecture, was given by Czerwi{\'n}ski
and Grytczuk~\cite{cg2008}.
We say that a runner $k$ is \emph{almost alone at time $t$} if there exists a $j\neq k$ such that
$$
 \| t(v_i-v_k)\|\geq \frac{1}{n+1}\;,
$$
for every $i\neq j,k$. If this case  we say that $j$ leaves  $k$ almost alone.

In~\cite{cg2008}, the authors showed that every runner is almost alone at some time. This means that Conjecture~\ref{conj:LRC} is true if we are allowed to
make one runner invisible, that is, there exists a time when all runners but one are far enough
from the origin.
\begin{theorem}[\cite{cg2008}]\label{thm:Inv}
For every $n\geq 1$ and every set of nonzero speeds $v_1,\dots ,v_{n}$, there exist a time $t\in
\TT$ such that the origin is almost alone at time $t$.

\end{theorem}

A similar result can be derived by using a model of dynamic circular interval graphs. By using this model we can show that either 
there is a runner alone at some time or at least four runners are almost alone at the same time.
\begin{theorem}\label{thm:my_invisible}
For every set of different speeds $v_1,\dots ,v_{n+1}$, there exist a time $t\in (0,1)$ such that, at time $t$, either there is a runner alone or four different runners are almost alone.
\end{theorem}

The paper is organized as follows. In Section~\ref{sec:correl} we compute the pairwise join
probabilities for the events $A_i$ and give a proof for Proposition~\ref{prop:variance}. As a
corollary of these results, we also prove Theorem~\ref{thm:2} and Theorem~\ref{thm:div}~(Subsection~\ref{subsec:2}). In Section~\ref{sec:weakest} we
introduce an approach on the problem based on dynamic circular interval graphs and prove
Theorem~\ref{thm:my_invisible}. Finally, in Section~\ref{sec:consec_lonely} we give some
conclusions , discuss some open questions and give a proof of the improved bound $1/(2n-3)$ when $2n-3$ is a prime which uses a combination of the ideas presented in this paper and a technique from \cite{bs2008}.

\section{Correlation among runners}\label{sec:correl}

In this section we want to study the pairwise join probabilities $\Pr(A_i\cap A_j)$, for every
$i,j\in [n]$.
Notice first, that, if $A_i$ and $A_j$ were independent events, then we would have  $\Pr(A_{i}\cap
A_{j})=4\delta^2$, since $\Pr(A_i)=2\delta$ for every $i\in [n]$. This is not true in the general
case, but, as we will see later on, some of these pairwise probabilities can be shown to be large enough.

For each ordered pair $(i,j)$ with $i,j\in [n]$, we define
\begin{align}\label{eq:eps}
\eps_{ij}=\left\{\frac{v_i}{\gcd (v_i,v_j)}\delta\right\} \;,
\end{align}
where $\gcd (v_i,v_j)$ denotes the greatest common divisor of $v_i$ and $v_j$ and $\{\cdot\}$ is the fractional part.

Let us also consider the function $f: [0,1)^2 \to \RR$, defined by
\begin{align}\label{eq:func_f}
f(x,y)=\min(x,y) + \max(x+y-1,0) -2xy\;.
\end{align}

The proofs of Propositions~\ref{prop:4delta2_coprime} and~\ref{prop:4delta2} below are based on  the proofs of lemmas~$3.4$ and~$3.5$ in  Alon and Ruzsa~\cite{ar1999}.
Let us start by studying the case when the speeds  $v_i$ and $v_j$ are coprime.
\begin{proposition}\label{prop:4delta2_coprime}
 Let $v_i$ and $v_j$ be coprime positive integers and $0<\delta <1/4$.
Then
$$
\Pr (A_i\cap A_j)= 4\delta^2+\frac{2f(\eps_{ij},\eps_{ji})}{v_i v_j}\;.
$$
\end{proposition}
\begin{proof}
 Observe that $A_i$ and $A_j$ can be expressed as the disjoint unions of intervals
$$A_i= \bigcup_{k=0}^{v_i-1} \left(\frac{k}{v_i}-\alpha,\frac{k}{v_i}+\alpha\right)
\qquad A_j=\bigcup_{l=0}^{v_j-1} \left(\frac{l}{v_j}-\beta,\frac{l}{v_j}+\beta\right) $$
where $\alpha=\delta/v_i$ and $\beta=\delta/v_j$ and we consider the elements in $[0,1)$ modulo $1$. We observe that $\alpha+\beta=\delta(v_i+v_j)/v_iv_j<1/2$ since $\delta <1/4$.

Denote by  $I=(-\alpha,\alpha)$ and $J=(-\beta,\beta)$. We have
\begin{align*}
\Pr (A_i\cap A_j) &= \Pr \left(\bigcup_{k\leq v_i,l\leq v_j} (I+k/v_i)\cap (J+l/v_j)\right) \\
&=\sum_{k\leq v_i,l\leq v_j} \Pr((I+k/v_i)\cap (J+l/v_j))\\
&=\sum_{k\leq v_i,l\leq v_j}\Pr ( I\cap (J+l/v_j-k/v_i))\\
&= \sum_{k=0}^{v_iv_j-1} \Pr \left(I\cap (J+k/v_jv_i)\right)\;,
\end{align*}
where in the last equality we used the fact that $\gcd  (v_i,v_j)=1$.

For each $-1/2<x<1/2$, define  $d(x) = \Pr (I\cap (J+x))$. Let us assume that $v_j<v_i$. We
can write $d(x)$ as follows (see
Figure~\ref{fig:d(x)}):
\begin{align*}
d(x)=\left\{\begin{array}{ll}
\beta+\alpha+ x,&x\in [-(\beta+\alpha),-(\beta-\alpha)]\\
2\alpha, & x\in  [-(\beta-\alpha),\beta-\alpha]\\
\beta+\alpha- x ,&x\in  [\beta-\alpha,\beta+\alpha]\\
0&\mbox{otherwise}
\end{array}\right.
\end{align*}

\begin{figure}[ht]
 \begin{center}
 \includegraphics[width=0.8\textwidth]{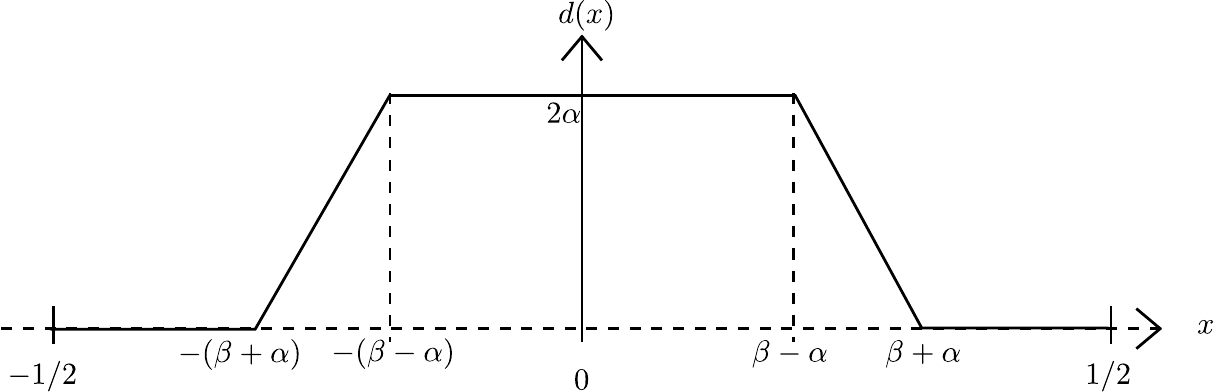}
 \end{center}
 \caption{Plot of $d(x)$ in $(-1/2,1/2)$.}
 \label{fig:d(x)}
\end{figure}

By symmetry, we have
$$
\frac{\Pr (A_i\cap A_j)}{2} = \frac{d(0)}{2} +  \sum_{1\le k\le (\beta+\alpha)v_iv_j}
d\left(\frac{k}{v_iv_j}\right)= \alpha + \sum_{1\le k\le (\beta+\alpha)v_iv_j} \min \left(2\alpha ,
\beta+\alpha-\frac{k}{v_iv_j}\right)\;.
$$

Write $\alpha v_iv_j= p+\eps_{ji}$ and $\beta v_iv_j= q+\eps_{ij}$, where $p$ and $q$ are integers and $0\leq \eps_{ji},\eps_{ij}<1$.

Observe that
$$
d\left(\frac{q-p}{v_iv_j}\right)v_iv_j=\left\{\begin{array}{ll}
2(p+\eps_{ji}), & \text{if }\eps_{ji}\leq \eps_{ij}\\
2p+\eps_{ji}+\eps_{ij}, &  \text{if }\eps_{ji}>\eps_{ij}\\
\end{array}\right. =  2p+ \eps_{ji}+\min(\eps_{ji},\eps_{ij})\;,
$$
and that
$$
d\left(\frac{q+p+1}{v_iv_j}\right)v_iv_j=\left\{\begin{array}{ll}
0, & \text{if }\eps_{ji}+\eps_{ij}\leq 1\\
\eps_{ji}+\eps_{ij}-1 ,&  \text{if }\eps_{ji}+\eps_{ij}> 1\\
\end{array}\right. =  \max(0,\eps_{ji}+\eps_{ij}-1)\;.
$$

Therefore,
\begin{align*}
 \frac{\Pr (A_i\cap A_j)}{2} v_iv_j&= p+\eps_{ji} + \sum_{1\le k \le p+q+\eps_{ji}+\eps_{ij}}
\min(2(p+\eps_{ji}),q+p+\eps_{ji}+\eps_{ij}-k)\\
&=p+\eps_{ji} +  \sum_{k=1}^{q-p-1}2(p+\eps_{ji}) + 2p+ \eps_{ji}+\min(\eps_{ji},\eps_{ij})\\
 & \qquad +  \sum_{k=q-p+1}^{p+q} (q+p+\eps_{ji}+\eps_{ij}-k) +  \max(0,\eps_{ji}+\eps_{ij}-1)\\
&= 2(p+\eps_{ji})(q+\eps_{ij}) +f(\eps_{ji},\eps_{ij})\;.
\end{align*}

Thus,
$$
\Pr (A_i\cap A_j)= \frac{2}{v_iv_j} \left( 2(p+\eps_{ji})(q+\eps_{ij}) +f(\eps_{ji},\eps_{ij})
\right) =
4\delta^2
+\frac{2f(\eps_{ji},\eps_{ij})}{v_iv_j} \;.
$$
\end{proof}

Proposition~\ref{prop:4delta2_coprime} can be easily generalized to pairs of speeds that are not coprime.
\begin{proposition} \label{prop:4delta2}
Let $v_i$ and $v_j$ be positive integers and $0<\delta <1$.
Then
$$
\Pr  (A_i\cap A_j)= 4\delta^2+\frac{2(\gcd (v_i,v_j))^2 f(\eps_{ji},\eps_{ij})}{v_iv_j}\;.
$$
\end{proposition}
\begin{proof}
 Consider $v_i'= \frac{v_i}{\gcd (v_i,v_j)}$ and $v_j'=\frac{v_j}{\gcd (v_i,v_j)}$.
Define $A'_i=  \{t\in \TT :\; \|tv'_i\|<\delta \}$ and $A'_j= \{t\in \TT :\; \|tv'_j\|<\delta \}$.
Observe that
$$
\Pr(A_i\cap A_j)=\Pr(A'_i\cap A'_j)\;.
$$
The proof follows by applying Proposition~\ref{prop:4delta2_coprime} to $v_i'$ and $v_j'$, which are
coprime.
\end{proof}

\begin{corollary}\label{cor:simple} For every $v_i$ and $v_j$ we have
\begin{equation}\label{eq:first bound}
\Pr(A_i\cap A_j)\ge 2\delta^2.
\end{equation}
Moreover,  if $v_j<v_i$, then
\begin{equation}\label{eq:second bound}
\Pr(A_i\cap A_j)\ge \frac{\gcd (v_i,v_j)}{v_i}2\delta\;.
\end{equation}
\end{corollary}
\begin{proof}
We observe that, for  $x,y\leq 1$, we have that $\min(x,y)\geq xy$ and thus $f(x,y)\geq -xy$. Therefore  Proposition~\ref{prop:4delta2} leads to  the following lower bound,
\begin{align*}
\Pr(A_i\cap A_j)&=  4\delta^2 + \frac{2 (\gcd (v_i,v_j))^2 f(\eps_{ij},\eps_{ji})}{v_i v_j}
\geq  4\delta^2 - \frac{2  (\gcd (v_i,v_j))^2\eps_{ij}\eps_{ji}}{v_i v_j}
\geq 2\delta^2\;.
\end{align*}
Moreover, if $v_j<v_i$, from the proof of Proposition~\ref{prop:4delta2_coprime} with
$v_i'=v_i/\gcd (v_i,v_j)$,
\begin{align*}
\Pr(A_i\cap A_j)&\geq d(0)=\frac{2\delta}{v_i'}=\frac{2 \gcd (v_i,v_j)\delta}{v_i}\;.
\end{align*}
\end{proof}

By using~\eqref{eq:first bound}, we can provide a first lower bound on the second moment of $X$,
\begin{align}\label{eq:var_1}
\E(X^2) =\sum_{i\neq j} \Pr(A_i\cap A_j) + \sum_{i=1}^n \Pr(A_i) \geq 2\delta^2 n(n-1)+ 2\delta n
\geq 2\delta n(\delta(n-1)+1)\;.
\end{align}
We devote the rest of this section to improve~\eqref{eq:var_1}.
Let us first show for which values is $f$  nonnegative.
\begin{lemma}\label{lem:f_positive}
 The function $f(x,y)$ is nonnegative in $[0,1/2]^2$ and in $[1/2,1)^2$.
\end{lemma}
\begin{proof}
 If $0\leq x,y\leq 1/2$, then $\min(x,y)\geq 2xy$, which implies $f(x,y)\geq 0$.

Moreover,
\begin{align*}
f(1-x,1-y)& = \min(1-x,1-y)+\max(1-x-y,0)-2(1-x-y+xy)\\
&=  \min(y,x)+\max(0,x+y-1)-2xy \\
&= f(x,y)\;.
\end{align*}

Therefore, we also have $f(x,y)\geq 0$ for all $1/2\leq x,y< 1$.
 \end{proof}

The following lemma shows that the error term of $\Pr(A_i\cap A_j)$ provided in
Proposition~\ref{prop:4delta2}, cannot be too negative if $v_i$ and $v_j$ are either close or
far enough from each other.
\begin{lemma}\label{lem:not_too_small}
Let $M\geq 2$ be an integer, $\gamma= M^{-1}>0$ and $v_j<v_i$. If either $(1-\gamma)v_i\leq v_j$ or
$ \gamma\delta v_i\geq v_j$, then
$$
\frac{(\gcd (v_i,v_j))^2 f\left(\eps_{ij},\eps_{ji}\right)}{v_i v_j} \geq -\gamma\delta^2\;.
$$
\end{lemma}
\begin{proof}
For the sake of simplicity, let us write $v_i/\gcd (v_i,v_j)=k\delta^{-1}+x$ and
$v_j/\gcd (v_i,v_j)=l\delta^{-1}+y$ with $k$ and $l$ being nonnegative integers and  $0\le x,y<\delta^{-1}$. In
particular, observe that
$\eps_{ij}=x\delta$ and $\eps_{ji}=y\delta$.
Moreover,  we can assume that $v_i$ and $v_j$ are such that $f\left(\eps_{ij},\eps_{ji}\right)$ is
negative, otherwise, there is nothing to prove.

We split the proof in the two different cases each consisting of some other subcases. Figure~\ref{fig:f(x,y)2} illustrates the subcase considered in each situation. Case

\fbox{ \textbf{Case A: ($\frac{v_i}{\gcd(v_i,v_j)}\geq (\gamma\delta)^{-1}$)}:}
  This case covers the case when $\gamma\delta v_i\geq v_j$, since $v_j/\gcd(v_i,v_j)\geq 1$ 
  and also the case when $(1-\gamma) v_i\leq v_j$ and $\frac{v_i}{\gcd(v_i,v_j)}\geq (\gamma\delta)^{-1}$.

\textbf{\underline{Subcase A.1} ($y\leq x$):} We have,
\begin{align*}
\frac{(\gcd (v_i,v_j))^2 f\left(\eps_{ij},\eps_{ji}\right)}{v_i v_j} &\geq \frac{(\gcd(v_i,v_j))^2 (\eps_{ji}-
2\eps_{ij}\eps_{ji})}{v_i v_j}
= \frac{(\gcd(v_i,v_j))^2 (y\delta- 2 xy\delta^{2})}{v_i v_j}
\ge \frac{\gcd (v_i,v_j)(1-2x\delta)}{v_i}\cdot\delta\;,
\end{align*}
where the last inequality holds from the fact that  $f\left(\eps_{ij},\eps_{ji}\right)<0$ and $y\leq
 v_j/\gcd (v_i,v_j)$.

Recall that $v_i/\gcd(v_i,v_j)\geq  (\gamma\delta)^{-1}= M\delta^{-1}$.
Observe also that, since $y\leq x$ and
$f(\eps_{ij},\eps_{ji})$ is negative, by Lemma~\ref{lem:f_positive} we have $\delta^{-1}/2 \leq
x<\delta^{-1}$. Therefore,
\begin{align*}
\frac{(\gcd(v_i,v_j))^2f\left(\eps_{ij},\eps_{ji}\right)}{v_i v_j} &\geq
\frac{\gcd (v_i,v_j)(1-2x\delta)}{v_i}\delta \geq
 \frac{1-2x\delta}{M}\delta^2 > -\gamma \delta^2\;.
\end{align*}

\textbf{\underline{Subcase A.2 ($y>x$):}} Here,
\begin{align*}
\frac{(\gcd(v_i,v_j))^2f\left(\eps_{ij},\eps_{ji}\right)}{v_i v_j} &\geq \frac{(\gcd(v_i,v_j))^2(\eps_{ij}-
2\eps_{ij}\eps_{ji})}{v_i
v_j}= \frac{(\gcd(v_i,v_j))^2 x(1- 2y\delta)}{v_i v_j}\delta
\geq  \frac{\gcd (v_i,v_j)(1-2y\delta)}{v_j}\cdot \frac{\delta}{M}\;,
\end{align*}
where the last inequality holds from the fact that, in this subcase, $Mx\leq M \delta^{-1}\leq v_i/\gcd(v_i,v_j)$.

As before, since $f(\eps_{ij},\eps_{ji})$ is negative, by
Lemma~\ref{lem:f_positive} we have $\delta^{-1}/2 \leq y<\delta^{-1}$ and $v_j/\gcd (v_i,v_j)\geq y$.
Therefore,
\begin{align*}
\frac{(\gcd(v_i,v_j))^2f\left(\eps_{ij},\eps_{ji}\right)}{v_i v_j} &\geq
\frac{\gcd (v_i,v_j)(1-2y\delta)}{v_j}\cdot\frac{\delta}{M}\geq
\frac{1-2y\delta}{y}\cdot \frac{\delta}{M}
> -\gamma \delta^2\;.
\end{align*}

 \fbox{\textbf{Case B ($(1-\gamma) v_i\leq v_j$ and $\frac{v_i}{\gcd(v_i,v_j)}\leq (\gamma\delta)^{-1}$)}:} 
 By Lemma~\ref{lem:f_positive} and since $\frac{v_i}{\gcd(v_i,v_j)}\leq (\gamma\delta)^{-1}$, we
can assume that either $k=l$, $y< \delta^{-1}/2$ and $x\geq \delta^{-1}/2$ (Subcases B.1 and B.2) or $k=l+1$, $y\geq \delta^{-1}/2$ and $x<\delta^{-1}/2$ (Subcases B.3 and B.4). 
In all these cases, $(1-\gamma) v_i\leq v_j$ implies
 \begin{equation}\label{eq:y}
 y\geq (1-\gamma)x-\gamma k\delta^{-1}\;.
 \end{equation}

\textbf{\underline{Subcase B.1} ($k=l$ and $x+y\leq\delta^{-1}$):} Since $x+y\leq \delta^{-1}$, then $\max(0,\eps_{ij}+\eps_{ji}-1)=0$.

By using $v_i/\gcd (v_i,v_j) = k\delta^{-1}+x$, $v_j/\gcd (v_i,v_j) =k\delta^{-1}+y\geq
y$ and the fact that $f\left(\eps_{ij},\eps_{ji}\right)<0$  we have
\begin{align*}
\frac{(\gcd(v_i,v_j))^2f\left(\eps_{ij},\eps_{ji}\right)}{v_i v_j} &= \frac{(\gcd(v_i,v_j))^2(\eps_{ji}-
2\eps_{ij}\eps_{ji})}{v_i v_j}
= \frac{(\gcd(v_i,v_j))^2 y(1-2x\delta)}{v_i v_j}\delta \geq
\frac{1-2x\delta}{k+x\delta}\delta^2\;.
\end{align*}
By combining \eqref{eq:y}  with  $x+y\leq
\delta^{-1}$, we get $x\leq \frac{1+\gamma k}{2-\gamma}\delta^{-1}$. Thus,
\begin{align*}
\frac{(\gcd(v_i,v_j))^2f\left(\eps_{ij},\eps_{ji}\right)}{v_i v_j} &\geq  \frac{1-2x\delta}{k+x\delta}
\delta^2 \geq  \frac{1-\frac{2(1+\gamma k)}{2-\gamma}}{k+\frac{1+\gamma k}{2-\gamma}}\delta^2=
-\gamma \delta^2\;,
\end{align*}
for each $k\geq 0$.

\textbf{\underline{Subcase B.2} ($k=l$ and $x+y\geq \delta^{-1}$):}

Now, $\max(0,\eps_{ij}+\eps_{ji}-1)=(x+y)\delta -1$. Then,
\begin{align*}
\frac{(\gcd(v_i,v_j))^2f\left(\eps_{ij},\eps_{ji}\right)}{v_i v_j} &
= \frac{(\gcd(v_i,v_j))^2(y\delta+ (x+y)\delta -1 -2xy\delta^2)}{v_i v_j}= -\frac{(\gcd(v_i,v_j))^2(1-2y\delta)(1-x\delta)}{v_i v_j}\delta\;.
\end{align*}
It remains to upper bound $g(x,y)=(1-2y\delta)(1-x\delta)=2\delta^2(\delta^{-1}/2-y)(\delta^{-1}-x)$ in the corresponding area. Observe that $\partial_x g(x,y) = -2\delta^2(\delta^{-1}/2-y)\leq 0$ for every $y\leq \delta^{-1}/2$. Thus, the local maximum of $g(x,y)$ is attained in the line $x+y=\delta^{-1}$. Observe that this case is covered by Case $B.1$. Following the steps of the previous case,
\begin{align*}
\frac{(\gcd(v_i,v_j))^2f\left(\eps_{ij},\eps_{ji}\right)}{v_i v_j} &\geq -\gamma \delta^2 \;.
\end{align*}

\textbf{\underline{Subcase B.3} ($k=l+1$ and $x+y\leq\delta^{-1}$):}

Now we have $\max(0,\eps_{ij}+\eps_{ji}-1)=0$ and $\min(\eps_{ij},\eps_{ji})=\eps_{ij}$.
Since $v_i/\gcd (v_i,v_j) = k\delta^{-1}+x\geq k\delta^{-1}$ and $v_j/\gcd (v_i,v_j) =(k-1)\delta^{-1}+y\geq y$, $v_j\geq (1-\gamma)v_i$ implies that $y\geq (1-\gamma)x-(\gamma k-1)\delta^{-1}$.

Then,
\begin{align*}
\frac{(\gcd (v_i,v_j))^2f\left(\eps_{ij},\eps_{ji}\right)}{v_i v_j} &= \frac{(\gcd (v_i,v_j))^2(\eps_{ij}-
2\eps_{ij}\eps_{ji})}{v_i v_j} = -\frac{(\gcd (v_i,v_j))^2 x\delta(2y\delta-1)}{v_i v_j}
\end{align*}

We aim to upper bound $g(x,y)=\delta x(2y\delta-1)=2\delta^2 x(y-\delta^{-1}/2)$ in the corresponding area. We have that $\partial_y g(x,y) = 2\delta^2x\geq 0$ for every $x\leq \delta^{-1}/2$. Again, the local maximum of $g(x,y)$ is attained in the line $x+y=\delta^{-1}$. A simple computation gives that $g(x,\delta^{-1}-x)$ attains its maximum in $x_0=\delta^{-1}/4$.

\textit{\underline{Subcase B.3.1} ($k\leq \gamma^{-1}/2$):}
In this case, the pair $(x_0,\delta^{-1}-x_0)$ does not lie in the area. Thus, the maximum is attained, when $x_0$ is minimized, that is in the point where $x+y\leq \delta^{-1}$ and $y= (1-\gamma)x-(\gamma k-1)\delta^{-1}$ meet. That is
$$
g(x,y)\leq g\left( \frac{\gamma k}{2-\gamma}\delta^{-1} , \frac{2-(k+1)\gamma}{2-\gamma}\delta^{-1} \right)=\frac{2-(2k+1)\gamma}{(2-\gamma)^2}\gamma k\;.
$$

Since $v_i/\gcd (v_i,v_j) \geq k\delta^{-1}$ and  $v_j/\gcd (v_i,v_j) \geq (k-1/2)\delta^{-1}$,
\begin{align*}
\frac{(\gcd (v_i,v_j))^2f\left(\eps_{ij},\eps_{ji}\right)}{v_i v_j} &\geq -\frac{2-(2k+1)\gamma}{(2-\gamma)^2(k-1/2)}\gamma \delta^2\\
 &\geq -\frac{2(2-3\gamma)}{(2-\gamma)^2} \gamma \delta^2 \geq \gamma \delta^2\;.
\end{align*}
where we used that $k\geq 1$.

\textit{\underline{Subcase B.3.2} ($k\geq \gamma^{-1}/2$):}
Using the global maximum of $g(x,\delta^{-1}-x)$ in $x_0=\delta^{-1}/4$, we have that for any $(x,y)$ in the area
$$
g(x,y)\leq g(\delta^{-1}/4,3\delta^{-1}/4)=1/8\;.
$$

Now $v_i/\gcd (v_i,v_j) \geq k\delta^{-1}\geq (\gamma\delta)^{-1}/2$ and  $v_j/\gcd (v_i,v_j) = (k-1)\delta^{-1}+y \geq  \delta^{-1}/2$,
\begin{align*}
\frac{(\gcd (v_i,v_j))^2f\left(\eps_{ij},\eps_{ji}\right)}{v_i v_j} &\geq -4\gamma\delta^2 \cdot\frac{1}{8}\geq - \gamma \delta^2\;.
\end{align*}

\textbf{\underline{Subcase B.4} ($k=l+1$ and $x+y\geq\delta^{-1}$):}
We have  $\max(0,\eps_{ij}+\eps_{ji}-1)=(x+y)\delta -1$ and $\min(\eps_{ij},\eps_{ji})=x$. Then,
\begin{align*}
\frac{(\gcd (v_i,v_j))^2f\left(\eps_{ij},\eps_{ji}\right)}{v_i v_j} &
=- \frac{(\gcd (v_i,v_j))^2(1-2x\delta)(1-y\delta)}{v_i v_j}\;.
\end{align*}

Now $g(x,y)=\delta (1-2x\delta)(1-y\delta)=2\delta^2(\delta^{-1}/2-x)(\delta^{-1}-y)$ and $\partial_x g(x,y) = -4\delta^2 (1-y\delta)\leq 0$ for every $\delta^{-1}/2\leq y\leq \delta^{-1}$. As usual, the local maximum of $g(x,y)$ is attained in the line $x+y=\delta^{-1}$. Since this case is already covered by Case $B.3$, we have
\begin{align*}
\frac{(\gcd(v_i,v_j))^2f\left(\eps_{ij},\eps_{ji}\right)}{v_i v_j} &\geq -\gamma \delta^2 \;.
\end{align*}

\end{proof}

\begin{figure}[ht]
 \begin{center}
 \includegraphics[width=0.8\textwidth]{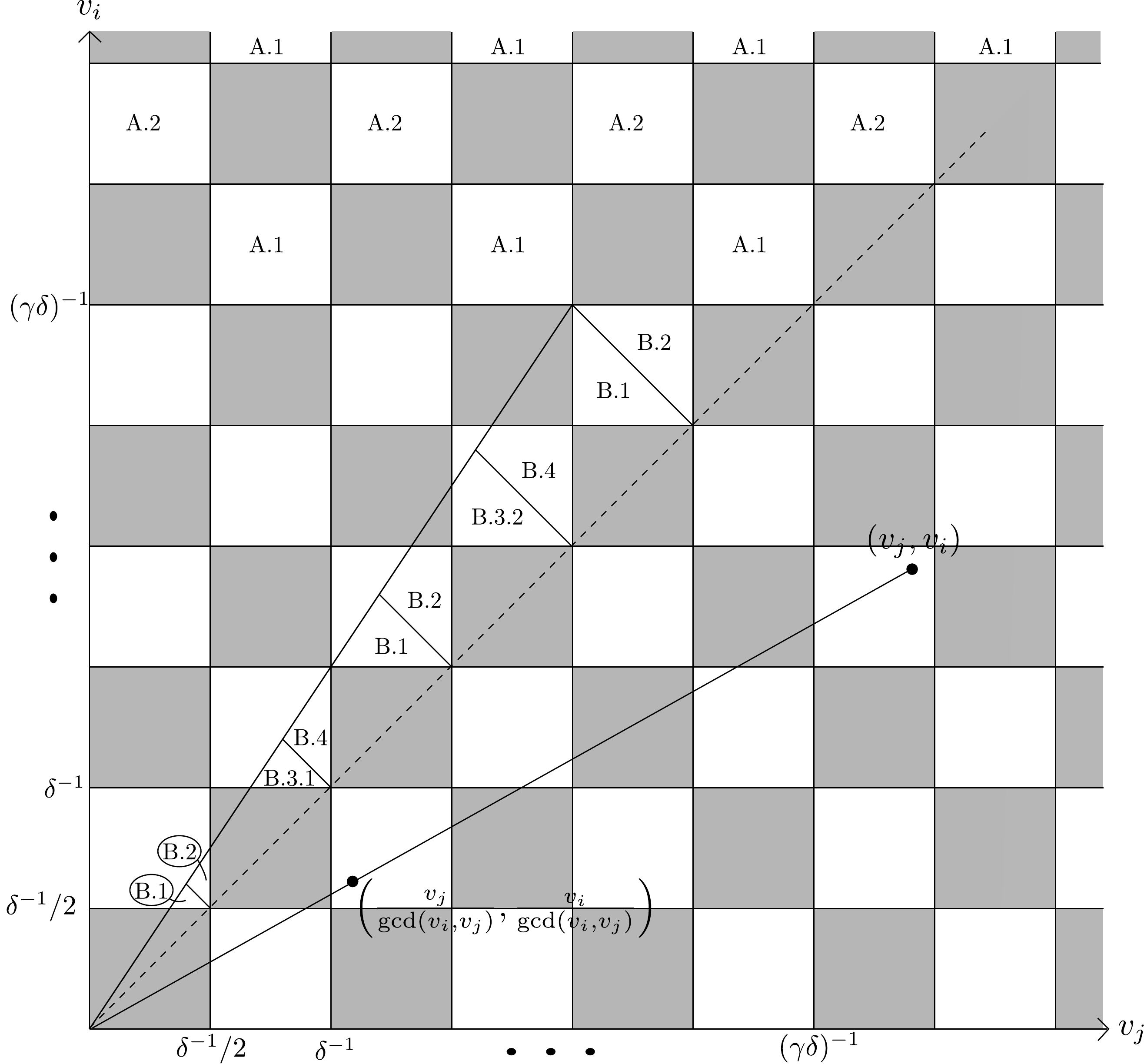}
 \end{center}
 \caption{Different cases in the proof of Lemma~\ref{lem:not_too_small}. Grey areas correspond to
positive values of $f\left(\eps_{ij},\eps_{ji}\right)$ according to Lemma~\ref{lem:f_positive}.}
 \label{fig:f(x,y)2}
\end{figure}

The following lemma shows that among a large set of positive numbers, there should be a pair of numbers
satisfying that they are either close or far enough from each other.
\begin{lemma}\label{lem:aux}
 For every $c>1$, $T>c$ and every set $x_1\geq \dots \geq x_{m+1}>0$ of
nonnegative numbers, with $m\ge \log_c{T}$, there is a pair  $i,j\in [m+1]$, $i<j$, such that
$$
\mbox{ either }\;\;  \frac{x_i}{x_j}\leq c \;\; \mbox{ or } \;\; \frac{x_i}{x_j}\geq T\;.
$$
\end{lemma}
\begin{proof}
 Suppose that for each pair  $i<j$ we have $x_i> c x_j$. In particular, for each $i\leq m$, we have
$x_i>c x_{i+1}$ and $x_1> c^m x_{m+1}\geq  T x_{m+1}$. Hence the second possibility
holds for $i=1$ and $j=m+1$.
\end{proof}

For any fixed $\eps>0$, we call a pair $i,j\in [n]$ \emph{$\eps$--good} if $\Pr(A_i\cap A_j)\geq (1-\eps)
4\delta^2 $.
Now we are able to improve the lower bound on the second moment of $X$ given in~\eqref{eq:var_1}.

\begin{proofof}[Proof of Proposition~\ref{prop:variance}]
	Recall that by~\eqref{eq:first bound}, for any pair $i,j\in [n]$, we have $\Pr(A_i\cap
A_j)\geq 2\delta^2$.
	We will show that at least a $\Omega\left(\frac{1}{\log \delta^{-1}} \right)$ fraction of
the pairs are $\eps$--good.

Lemma~\ref{lem:not_too_small} with $M=\gamma^{-1}=\lceil (2\eps)^{-1}\rceil$ implies that
every pair $v_j<v_i$ with either $v_i/v_j\leq (1-\gamma)^{-1}$ or $v_i/v_j\geq (\gamma \delta)^{-1}$
is a $(\gamma/2)$--good pair, and thus, also an $\eps$--good pair.

Consider the graph $H$ on the vertex set $V(H)=[n]$, where $ij$ is an edge if and only if $ij$ is
$\eps$--good. Using Lemma~\ref{lem:aux}, with $c=(1-\gamma)^{-1}$ and $T= (\gamma \delta)^{-1}$ we
know that there are no independent sets of size larger than
$m=\lceil\log_c(T)\rceil=\lceil\frac{\log{\delta^{-1}}}{\log c}+\log_c{\gamma^{-1}}\rceil$. 
Since $\delta\to 0$ when $n\to +\infty$, if $n$ is large enough, there exists some constant $c'_\eps$ that depends only on $\eps$ such that $m \geq \frac{\log{\delta^{-1}}}{c'_\eps}+1$.
Thus, the complement of $H$, $\overline{H}$, has no clique of size $m$.
By the Erd\H os--Stone theorem (see~\cite{es1946}), $|E(\overline{H})|\leq
(1+o(1))\frac{m-2}{m-1}\frac{n^2}{2}$, which implies that there are
$$
|E(H)|\geq (1+o(1))\frac{n^2}{2(m-1)}=(1+o(1))\frac{c'_\eps n^2}{2\log{\delta^{-1}}}\;,
$$
$\eps$--good unordered pairs.

Now, we are able to give a lower bound on the second moment,
\begin{align*}
\E(X^2) &= \sum_{ij\; \eps\text{--good}} \Pr(A_i\cap A_j)+\sum_{ij\;\text{non } \eps\text{--good}}
\Pr(A_i\cap A_j) + \sum_{i=1}^n \Pr(A_i) \\
&\geq (1-\eps)4\delta^2\frac{c'_\eps n^2}{\log{\delta^{-1}}}+ 2\delta^2
\left(n(n-1)-\frac{c'_\eps n^2}{\log{\delta^{-1}}}\right)+ 2\delta n\\
&=  (1-\eps)4\delta^2\frac{c'_\eps n^2}{\log{\delta^{-1}}}+ 2\delta^2
\left(1-\frac{c'_\eps}{\log{\delta^{-1}}}-\frac{1}{n}\right)n^2+ 2\delta n\\
&\geq 2\delta n\left(\delta\left(1+\frac{c_\eps}{\log{\delta^{-1}}}\right)n+ 1\right)\;,
\end{align*}
for some $c_\eps$ that depends only on $\eps$.
	
\end{proofof}

Next, we show some applications of our bounds that extend some known results.

\subsection{Improving the gap of loneliness}\label{subsec:2}

In this subsection we show how to use the result of Proposition~\ref{prop:4delta2} on the pairwise join probabilities to prove Theorem~\ref{thm:2}. To this end we will use the following Bonferroni--type inequality due to Hunter \cite{hunter1976} (see also Galambos and Simonelli  \cite{gs1996}) that slightly improves the union bound in the case where the events are not pairwise disjoint.

\begin{lemma}[Hunter \cite{hunter1976}]\label{lem:bonf_local}
For any tree $T$ with vertex set $V(T)=[n]$, we have
\begin{align}\label{eq:Bonferroni_book2}
 \Pr\left(\bigcap_{i=1}^n  \overline{A_i}\right)&\geq 1- \sum_{i=1}^n \Pr(A_i) + \sum_{ij\in E(T)}
\Pr(A_i\cap A_j)\;.
\end{align}
\end{lemma}

As we have already mentioned, $\Pr(A_i)=2\delta$. Thus, it remains to select a tree $T$ that
maximizes $\sum_{ij\in E(T)} \Pr(A_i\cap A_j)$.

\begin{lemma}\label{prop:tree}
 For each $\eps'>0$, there exists a tree $T$ on the set of vertices $[n]$ such
that
$$
\sum_{ij\in E(T)} \Pr(A_i\cap A_j) \geq (1-\eps') 4\delta^2 n\;.
$$
\end{lemma}
\begin{proof}
	Recall that Proposition~\ref{prop:4delta2} states that
	$$
	\Pr(A_i\cap A_j) = 4\delta^2 + \frac{2(v_i,v_j)^2f(\eps_{ij},\eps_{ji})}{v_iv_j}.
	$$
	Set $M$ to be the largest integer satisfying $M=\gamma^{-1}< \lceil (2\eps')^{-1} \rceil$.
	We will construct a large forest $F$ on the set of vertices $[n]$, where all the edges
$ij\in E(F)$ are $\eps'$--good. In particular they will satisfy,
	$$
	\Pr(A_i\cap A_j)\geq  (4-2\gamma)\delta^2 \geq (1-\eps')4\delta^2\;.
	$$

Let us show how to select the edges of the forest by a procedure. Set $S_0=[n]$ and
$E_0=\emptyset$. In the $k$-th step, we select different $i,j\in S_{k-1}$ such that either $v_i/v_j
\leq (1-\gamma)^{-1}$ or $v_i/v_j \geq (\gamma\delta)^{-1}$, and set $E_k=E_{k-1}\cup \{ij\}$,
$S_{k}=S_{k-1}\setminus \{i\}$. If no such pair exists, we stop the procedure.

Let $\tau$ be the number of steps that the procedure runs before being halted. By
Lemma~\ref{lem:aux} with $c=(1-\gamma)^{-1}$ and $T=(\gamma\delta)^{-1}$ we can always find such an
edge $ij$, provided that the set $S_k$ has size at least $\log_c{(\gamma\delta^{-1})}$. Thus
$\tau\geq n-\log_{c}{T}$. Since the size of the sets $E_k$ increases exactly by
one at each step, we have $|E_\tau |\geq n-\log_c{T}= n-O(\log{\delta^{-1}})= (1-o(1))n$. Besides, by construction
$E_\tau$ is an acyclic set of edges: since we delete one of the endpoints of each selected edge from the set
$S_k$, $E_{\tau}$ induces a $1$-degenerate graph or equivalently, a forest.

By Lemma~\ref{lem:not_too_small}, for each edge $ij$ in $E_\tau$ we have
$$
\Pr(A_i\cap A_j)\geq (4-2\gamma)\delta^2\;.
$$
Therefore we can construct a spanning tree $T$ on the vertex set $[n]$, that contains the forest
$F$ and thus satisfies
$$
 \sum_{ij\in E(T)} \Pr(A_i\cap A_j) \geq \left(1-o(1)\right)(4-2\gamma) \delta^2 n \geq
(1-\eps') 4\delta^2 n\;,
$$
if $n$ is large enough.

\end{proof}

Let us proceed to prove Theorem~\ref{thm:2}.
\begin{proofof}[Proof of Theorem~\ref{thm:2}]
Given an $\eps>0$, by Lemma~\ref{eq:Bonferroni_book2} and Lemma~\ref{prop:tree} with $\eps'=\eps/2$, we have
\begin{align*}
\Pr\left(\bigcap_{i=1}^n \overline{A_i}\right)&\geq 1- \sum_{i=1}^n \Pr(A_i) +\sum_{ij\in E(T)}
\Pr(A_i\cap A_j)\nonumber\\
 &\geq 1-2\delta n(1-  2 (1-\eps')\delta)\;.
\end{align*}
The expression above is strictly positive for
$$
\delta \le  \frac{1}{2n-2+2\eps'} =\frac{1}{2n-2+\eps}\;,
$$
and the theorem follows.
\end{proofof}

Theorem \ref{thm:div} follows from the following Corollary.

\begin{corollary}\label{cor:divergence}
 For every sufficiently large $n$  and every set of nonzero speeds $v_1,\dots
,v_{n}$, such that there exists a tree $T$,
\begin{align}\label{eq:condition}
\sum_{ij\in E(T)} \frac{\gcd (v_i,v_j)}{\max(v_i,v_j)}=c\;,
\end{align}
then there exists a time  $t\in \TT$ such that
$$
 \| tv_i\|\geq\frac{1}{2(n-c)}\;,
$$
for every $i\in [n]$.
In particular, if $\sum_{i=2}^{n} \frac{1}{v_i}=c$ the same conclusion follows.
\end{corollary}
\begin{proof} By using inequality~\eqref{eq:second bound} we have
$$
\sum_{ij\in E(T)} \Pr(A_i\cap A_j) \geq  2\delta \sum_{ij\in E(T)}
\frac{\gcd (v_i,v_j)}{\max(v_i,v_j)}=2c\delta
$$
Using Lemma~\ref{eq:Bonferroni_book2} in a similar way as in the proof of Theorem~\ref{thm:2}, we
obtain that $\Pr\left(\bigcap_{i=1}^n \overline{A_i}\right)>0$ for
every $\delta < \frac{1}{2(n-c)}$.

The last part of the corollary follows by considering $T$ to be the star with center in the
smallest of the speeds.
\end{proof}

\section{Weaker conjectures and interval graphs}\label{sec:weakest}
In this section we give a proof for Theorem~\ref{thm:my_invisible}.
The following weaker conjecture has been proposed by Spencer\footnote{Transmitted to the authors by
Jarek Grytczuk.}.
\begin{conjecture}[Weak Lonely Runner Conjecture]\label{conj:weak}
For every $n\geq 1$ and every set of different speeds $v_1,\dots ,v_{n}$, there exist a time
$t\in\RR$ and a runner $j\in [n]$, such that
$$
 \| t(v_i-v_j)\|\geq \frac{1}{n}
$$
for every $i\neq j$.
\end{conjecture}

For every set $S\subseteq [n]$, we say that $S$ is \emph{isolated at time $t$} if,
\begin{align}
 \| t(v_i-v_j)\|\geq \frac{1}{n} \qquad \text{for each $i\in S$, $j\in V\setminus S$.}
\end{align}
Observe that $S=\{i\}$ is isolated at time $t$, if and only $v_i$ is lonely at time $t$.

To study the appearance of isolated sets, it is convenient to define a dynamic graph $G(t)$, whose
connected components are sets of isolated runners at time $t$. For each $1\leq i\leq n$ and $t\in
\TT$, define the following
dynamic interval in $\TT$ associated to the $i$-th runner,
$$
I_i(t) = \left\{x\in \TT\;: \{ x-tv_i \}< \frac{1}{n}\right\}\;.
$$
In other words, $I_i(t)$ is the interval that starts at the position of the $i$-th
runner at time $t$  and has length $\frac{1}{n}$.

Now we can define the following dynamic circular interval graph  $G(t)=(V(t),E(t))$. The vertex set
$V(t)$ is composed by $n$ vertices $u_i$ that correspond to the set of runners, and two vertices
$u_i$ and $u_j$ are connected if and only if $I_i(t)\cap I_j(t)\neq \emptyset$ (see
Figure~\ref{fig:interval}).

\begin{figure}[ht]
 \begin{center}
 \includegraphics[width=\textwidth]{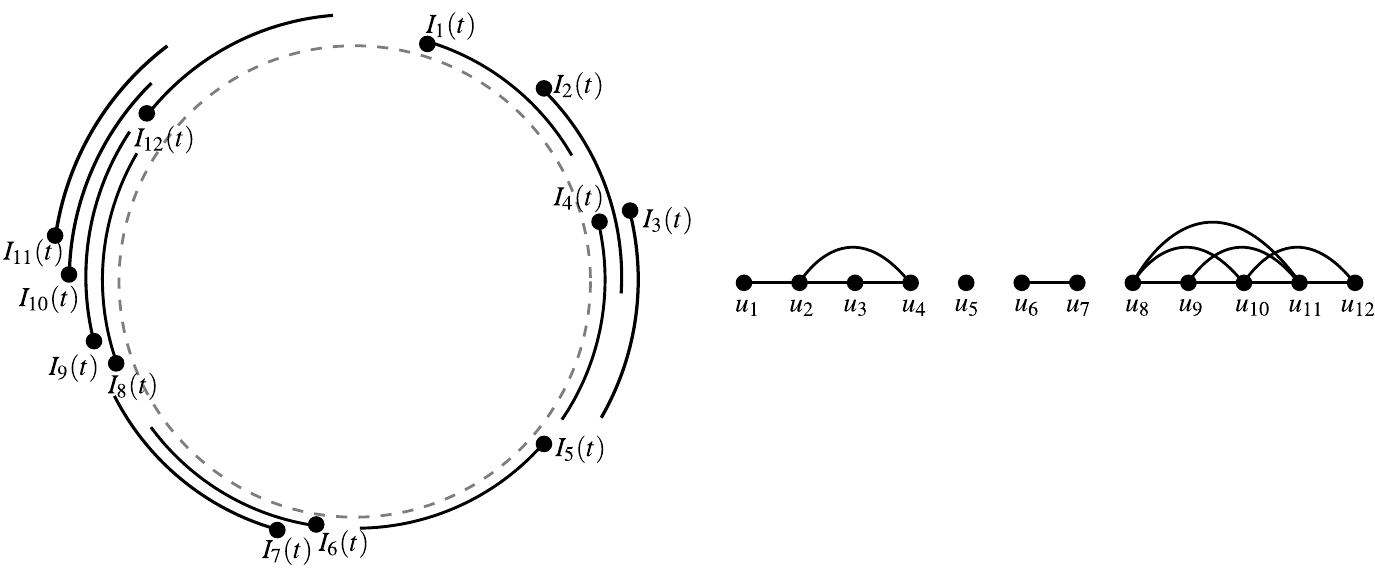}
 \end{center}
 \caption{An instance of the graph $G(t)$.}
 \label{fig:interval}
\end{figure}

\begin{observation}
 The graph $G(t)$ satisfies the following properties,
\begin{enumerate}
 \item $G(0) = K_n$.
 \item Each connected component of $G(t)$, correspond to an isolated set of
runners at time $t$.
 \item If $u_i$ is isolated in $G(t)$, then $v_i$ is alone at time $t$.
 \item All the intervals have the same size, $|I_i(t)|=1/n$, and thus, $G(t)$ is a unit circular
interval
graph.
\end{enumerate}
\end{observation}

We can restate the Lonely Runner Conjecture in terms of the dynamic interval graph $G(t)$.
\begin{conjecture}[Lonely Runner Conjecture]
For every $i\leq n$ there exists a time $t$ such that $u_i$ is isolated in $G(t)$.
\end{conjecture}

Let $\mu$ be the uniform measure in the unit circle. For every subgraph $H\subseteq G(t)$ we define $\mu(H) = \mu(\cup_{u_i\in
V(H)} I_i(t))$, the length of the arc occupied by the intervals corresponding to $H$.
Notice that, if $H$   contains an edge, then
\begin{align}\label{eq:property_H}
\mu(H)<\frac{|V(H)|}{n}\;,
\end{align}
since the intervals $I_i(t)$ are closed in one extreme but open in the other one.
If $H$  consists of isolated vertices, then~\eqref{eq:property_H} does not hold.

The dynamic interval graph $G(t)$ allows us to prove a very weak version of the conjecture. Let us assume
that $v_1<v_2<\cdots <v_n$.

\begin{proposition}\label{prop:weakest}
 There exist a time $t\in \RR$ and a nonempty subset $S\subset [n]$ such that $S$ is isolated at
time $t$.
\end{proposition}
\begin{proof}
Let $t$ be the minimum number for which the equation $tv_n-1= tv_1-1/n$ holds.
This is the first time that the slowest runner $v_1$ is at distance exactly $1/n$ ahead from the fastest runner $v_n$.

For the sake of contradiction, assume that there is just one connected
component of order $n$.
Note that $u_1 u_n \notin E(G(t))$ and since $G(t)$ is connected, there exists a path in
$G(t)$ connecting $u_1$ and $u_n$.
By~\eqref{eq:property_H}, we have $\mu(G)< 1$. Thus, there is a point $x\in \TT$ such that
$x\notin I_i(t)$ for every $i\in [n]$.

Observe that, at   time $t$, all the intervals $I_i(t)$ are sorted in increasing order around
$\TT$.
Let $\ell\in [n]$ be such that $x>\{tv_{\ell}\}$ and $x<\{tv_{\ell+1}\}$. Then, $\{u_1,\dots,
u_{\ell}\}$ and $\{u_{\ell+1},\dots, u_n\}$ are in different connected components since $u_1 u_n,
u_{\ell} u_{\ell+1} \notin E(G(t))$.
\end{proof}
Observe that, if one of the parts in Proposition~\ref{prop:weakest} consists of a singleton, say
$S=\{i\}$, then Conjecture~\ref{conj:weak} would be true.

Let us show how to use the dynamic graph to prove an invisible lonely runner
theorem, similar to Theorem~\ref{thm:Inv}.
\begin{proposition}\label{prop:invisibles}
	There exists $t\in\TT$ such that $G(t)$ has either at least one isolated vertex or at least four vertices of degree one.
\end{proposition}
\begin{proof}
	Define $Y : \TT \to \mathbb{N}$  by,
	$$
	Y(t)=|E(G(t))| \;.
	$$
Let $t\in \TT$ be chosen uniformly at random. Then $Y(t)$ is a random variable over $
\left\{0,1,\dots, \binom{n}{2}\right\}$.
	We will show that $\E(Y(t))= n-1$. If we are able to do so, since $Y(t)$ is not constant, by a first moment
argument, we know that there exists a time $t_0$ for which $Y(t_0)\leq  n-2$. Then, denoting by $d_i$ the degree of $u_i$, we have
$$
\sum_{i=1}^n d_i \leq 2(n-2)\;.
$$
Suppose that there are no isolated vertices. Then $d_i>0$ for $i$ and this ensures the existence of at least $4$ vertices of degree one, concluding the proof of the proposition.

Now, let us show that $\E(Y(t))\leq (n-1)$. We can write
$Y(t)=\sum_{i<j} Y_{ij}(t)$, where $Y_{ij}(t)=1$ if $u_i$ and $u_j$ are connected at time $t$
and $Y_{ij}(t)=0$ otherwise. Then $\E(Y(t))= \sum_{i<j} \E(Y_{ij}) = \sum_{i<j} \Pr(I_i(t)\cap
I_j(t)\neq \emptyset)$.
	For the sake of simplicity when computing $\Pr(I_i(t)\cap I_j(t)\neq \emptyset)$, we can
assume that $v_i=0$. Since the intervals are half open, half closed, we have $\Pr(I_i(t)\cap
I_j(t)\neq \emptyset)=2/n$, no matter the value of $v_j$.

Finally,
	$$
	\E(Y(t))= \sum_{i<j} \frac{2}{n} = \binom{n}{2}\frac{2}{n}=n-1 \;.
	$$
\end{proof}

In the dynamic interval graph setting, an invisible runner is equivalent to a vertex $u$ with a
neighbor of degree one, say $v$. If $u$ is removed, then $v$ becomes isolated in $G(t)$ and thus,
alone in the runner setting. Thus, Theorem~\ref{thm:my_invisible} is a direct corollary of
Proposition~\ref{prop:invisibles}.

\section{Concluding remarks and open questions}\label{sec:consec_lonely}

\textbf{1.}
In Proposition~\ref{prop:variance} we gave a lower bound for $\E(X^2)$. We believe that this proof
can be adapted to show that $\E(X^2)$ is larger.
\begin{conjecture}
	For every set of different speeds $v_1,\dots, v_n$, and every $\delta<1$, we
have
	$$
	\E(X^2)\geq (1+o(1))4\delta^2 n^2+ 2\delta n\;.
	$$
\end{conjecture}
The proof of this conjecture relies on showing that either most pairs are $\eps$--good or the
contribution of the positive error terms is larger than the contribution of the negative ones. On
the other hand, notice that it is not true that $\E(X^2)$ is $O(\delta^2 n^2)$. For the set of
speeds in~\eqref{eq:speeds}, Cilleruelo~\cite{cille} has shown that
\begin{align}\label{eq:cille}
	\E(X^2)= (1+o(1))\frac{12}{\pi^2}\delta n\log{n}\;,
\end{align}
which is a logarithmic factor away from the lower bound in Proposition~\ref{prop:variance},
when $\delta= \Theta(n^{-1})$. It is an open question whether~\eqref{eq:cille} also holds as an
upper bound for $\E(X^2)$.

\textbf{2.}
Ideally, we would like to estimate the probabilities $\Pr(\cap_{i\in S}A_i)$, for every set
$S\subseteq [n]$ of size $k$. In general, it is not easy to compute such probability. As
in~\eqref{eq:cille}, the sum of the $k$--wise join probabilities are not always of the form
$O(\delta^k n^k)$. However it seems reasonable to think that, for every set $S$,
we have
$$
\Pr(\cap_{i\in S}A_i) \geq c_k \delta^k\;,
$$
where $c_k$ depends only on $k$. Moreover, we know that $c_k\leq 2^k$, since this is the case when
the speeds $\{v_i\}_{i\in S}$ are rationally independent and the conjecture holds   (see e.g.
Horvat and Stoffregen \cite{horvat2011}.)

The inequality~\eqref{eq:first bound} shows that $c_2=2$.
In general, we also conjecture that
$$
c_k= (1+o_k(1))2^k\;.
$$

\textbf{3.}
Below we give a short proof of the result by Chen and Cusick \cite{cc1999} improving the bound on the lonely runner problem with $n$ runners when $p=2n-3$ is a prime number.

\begin{proposition} If $2n-3$ is a prime number then, for every set of $n$ speeds, there exists a time $t\in \RR$ such
that
  $$
\| tv_i\|\geq\frac{1}{2n-3}\;,
  $$
for every $i\in [n]$.

\end{proposition}

\begin{proof} Let $p=2n-3$. We may assume $p>7$. For a positive integer $x$ let $\nu_p(x)$ denote  the smallest power of $p$ in the $p$--adic expansion of $x$. Set $m=\max_i \nu_p(v_i)$ and $N=p^{m+1}$. We consider the problem in the cyclic group $\ZZ_N$. We will show   that there is $t\in \ZZ_N$ such that the circular distance to the origin of $tv_i$ to $0$ in $\ZZ_N$, denoted by $\|tv_i\|_N$, is at least $p^m=N/p$. This clearly implies that $\| tv_i\|\ge 1/p$ and proves the Proposition.

If $m=0$ then we have $\|v_i\|_p\ge 1$ for each $i$ and there is nothing to prove, so assume $m>0$. 

For each positive integer $x$ we denote by $\pi_j(x)$ the coefficient of $p^j$ in the $p$--adic expansion of the representative in $\{0,1,\ldots ,N-1\}$ of $x$ modulo $N$. We seek a certain $t$ such that,  for each $i$,   $\pi_m(tv_i)$   does not belong to $\{0,p-1\}$. This  implies that $\|tv_i\|_N\ge p^m$ for each $i$, which is our goal.

Partition the set $V=\{v_1,\ldots ,v_n\}$ of speeds into the sets $V_j=\{v_i: \nu_p(v_i)=j\}$, $j=0,1,\ldots ,m$. Since $\gcd (V)=1$ we have $A_0\neq \emptyset$. By the definition of $m$, we also have $V_m\neq \emptyset$. We consider two cases:

{\it Case 1.\/} $|V_0|<n-1$. This implies $|V_j|<n-1$ for each $j$.

For each  $\lambda\in \{1,\ldots ,p-1\}$ and each $v_i\in V_m$ we have $\pi_m(\lambda v_i)=\lambda \pi_m(v_i) \pmod p$, which is nonzero. By pigeonhole, there is $\lambda$  such that $\pi_m (\lambda v_i)\not\in \{ 0,p-1\}$ for each $v_i\in V_m$. (Actually, for speeds in $V_m$ it is enough that $\pi_m (\lambda v_i)\neq 0$.)

Suppose that there is $\lambda_{j+1} \in \{0,1,\ldots ,p-1\}$ such that $\pi_m (\lambda_{j+1} v_i)\not\in \{0,p-1\}$ for each $v_i\in  V_{j+1}\cup V_{j+2}\cup \cdots \cup V_m$ and some  $j+1\le m$. Since $|V_j|< n-1$, there is $\mu\in \{0,1,\ldots ,p-1\}$ such that 
$$
\pi_m((1+\mu p^{m-j})v_i)=\pi_m(v_i)+\mu\pi_j(v_i) \pmod p
$$
does not belong to $\{0,p-1\}$ for each $v_i\in V_j$. Moreover,  for each $v_i\in \cup_{l>j}V_l$,
$$
\pi_m((1+\mu p^{m-j})\lambda_{j+1}v_i)=\pi_m(\lambda_{j+1}v_i).
$$ 
Therefore, by setting $\lambda_j=(1+\mu p^{m-j})\lambda_{j+1}$, we have $\pi_m(\lambda_jv_i)$ not in $\{0,p-1\}$ for each $v_i\in \cup_{l\ge j}V_l$. Hence the sought multiplier is $t=\lambda_0$.

{\it Case 2.\/} $|A_0|=n-1$. Thus $V=V_0\cup V_m$ with $|V_m|=1$, say $V_m=\{ v_n\}$. For each $\lambda\in \ZZ_N^*$ we have $\|\lambda v_n\|_N\ge p^m$. We may assume that $v_1=1 \pmod N$. Let 
$$
A_i=\{\lambda\in \ZZ_N^*: \|\lambda v_i\|_N<p^m\}
$$
denote the set of bad multipliers for $v_i$. By choosing a multiplier in $\ZZ_N^*$ uniformly at random  we have 
$$
\Pr (A_i)=2p^{m-1}(p-1)/p^m(p-1)=2/p.
$$

If we show that $\Pr (A_1\cap A_i)> 2/p(p-1)$ for each $i=2,\ldots ,n-1$ then, by using Hunter's inequality \ref{lem:bonf_local}, 
$$
\Pr (\cap_{i=1}^{n-1} \overline{A_i})\ge 1-\Pr (\cup_{i=1}^n A_i)\ge 1-\left(\sum_{i=1}^{n-1} \Pr (A_i)-\sum_{i=2}^{n-1} \Pr (A_1\cap A_j)\right)>1-\left(1+\frac{1}{p}-\frac{p-1}{2}\frac{2}{p(p-1)}\right)=0,
$$
which implies  that there is $\lambda\in \ZZ_N^*$ such that $\|\lambda v_i\|_N\ge p^m$ for all $i$, concluding the proof.

Suppose $m\ge 2$. Consider the set $C=[0,p^m]\cap\ZZ_N^*$, so that $\lambda\in A_i$ if and only if $\lambda v_i\in C\cup (-C)$. The pairwise disjoint sets 
$$
C_j=\{jp+1,2(jp+1),\ldots ,(p-1)(jp+1)\},\;  j=0,1,\ldots ,p^{m-2}+1,
$$ 
satisfy
$$
C_j\subset C, \;\;  (C_j-C_j)\setminus \{0\}=  C_j\cup (-C_j)\;\;\text{and}\;\;    (C_j-C_j)\cap  (C_{j'}-C_{j'})=\{0\}, j\neq j'.
$$
The first   inclusion holds because no two elements in $C_j$ are congruent modulo $p$ and the largest element in $C_j$ is at most $(p-1)(p^{m-1}+p+1)=p^{m}-p^{m-1}+p^2-1< p^{m}$ if $m\ge 3$ and $(p-1)(p+1)=p^2-1$ if $m=2$. The  last two equalities   clearly hold.
 
Fix $i\in \{2,\ldots n-1\}$. We have $\lambda v_i\in \ZZ_N^*$ for each $\lambda\in C_j$ and each $C_j$.  By pigeonhole, if $\lambda v_i\not\in C\cup (-C)$ for each $\lambda\in C_j$ then there  are distinct   $\lambda, \lambda'\in C_j$ such that $\lambda v_i, \lambda'v_i\in C+kp^m$ for some $k\in \{1,2,\ldots ,p-2\}$. Therefore $(\lambda-\lambda')v_i\in C\cup (-C)$. Thus, for each $j$ there is $\mu \in (C_j-C_j)\setminus \{0\}$ such that both $\mu v_i, -\mu v_i\in C\cup (-C)$. Since $(C_j-C_j)\cap (C_{j'}-C_{j'})=\{ 0\}$, we have 
$$
\Pr (B_1\cap B_i)\ge \frac{2(p^{m-2}+2)}{p^{m-1}(p-1)}>\frac{2}{p(p-1)}.
$$

It remains to consider the case $m=1$. In this case  the same argument as above with an only $C_0=\{1,2,\ldots ,p-1\}$ suffices to give  $\Pr (B_1\cap B_i)\ge 1/(p-1)$. 
\end{proof}

\begin{acknowledgement}
The authors would like to thank Jarek Grytczuk for suggesting the use of dynamic graphs in the last part of the paper.
This research was supported by the Spanish Research Council under project MTM2008-06620-C03-01 and by the Catalan Research Council under grant
2014 SGR 1147.
\end{acknowledgement}

\bibliography{lonely}
\bibliographystyle{amsplain}

\end{document}